\documentclass{amsart}

\usepackage{amsmath,amsthm,amsopn,amstext,amscd,amsfonts,amssymb,mathrsfs,mathtools}
\usepackage{dsfont}
\usepackage{comment}
\usepackage{xcolor}
\usepackage{float}
\usepackage[active]{srcltx}
\usepackage{graphicx, mathdots}

  \usepackage{mathtools}

\newtheorem*{conj}{\sc Conjecture}

\newtheorem{proposition}{\sc Proposition}[section]

\usepackage[active]{srcltx}

\usepackage{graphicx, epsfig, subfig}
\usepackage{lscape}
\usepackage{rotating}
\usepackage{caption}
\usepackage{multicol}
\usepackage{colortbl}
\usepackage{nicematrix}

\newcommand*\pPq[5]{%
 \begingroup
 \begingroup\lccode`~=`,
   \lowercase{\endgroup\def~}{\pFcomma\mkern\pFqskip}%
 \mathcode`,=\string"8000
 {}_{#1}\phi_{#2}\left(\left.\genfrac..{0pt}{}{#3}{#4}\right|\,#5\right)%
 \endgroup
}

\newmuskip\pFqskip
\mathchardef\pFcomma=\mathcode`, 

\makeatletter
\def\BState{\State\hskip-\ALG@thistlm}
\makeatother

\def\downbar#1{
\setbox10=\hbox{$#1$}
            \dimen10=\ht10 \advance\dimen10 by 2.5pt
            \ifdim \dimen10<15pt 
               \advance\dimen10 by -0.5pt
               \dimen11=\dimen10
               \advance\dimen10 by 2.5pt
               \lower \dimen11
            \else \lower \ht10 \fi
            \hbox {\hskip 1.5pt \vrule height \dimen10 depth \dp10}}
\def\upbar#1{
\setbox10=\hbox{$#1$}
            \dimen10=\ht10 \advance\dimen10 by \dp10 \advance\dimen10 by 2.5pt
            \ifdim \dimen10<15pt 
                \advance\dimen10 by 2pt \fi
            \raise 2.5pt \hbox {\hskip -1.5pt \vrule height \dimen10}}

\begin{document}
\title[A counterexample to a conjecture of M. Ismail]{A counterexample to a conjecture of M. Ismail}
\author{K. Castillo}
\address{CMUC, Department of Mathematics, University of Coimbra, 3001-501 Coimbra, Portugal}
\email{ kenier@mat.uc.pt}
\author{D. Mbouna}
\address{University of Almer\'ia, Department of Mathematics, Almer\'ia, Spain}
\email{mbouna@ual.es}

\date{\today}
\maketitle
\begin{abstract}
In an earlier work [K. Castillo et al., J. Math. Anal. Appl. {\bf 514} (2022) 126358], we give positive answer to the first, and apparently more easy, part of a conjecture of M. Ismail concerning the characterization of the continuous $q$-Jacobi polynomials, Al-Salam-Chihara polynomials or special or limiting cases of them. In this note we present an example that disproves the second part of such a conjecture, and so this issue is definitively closed.
\end{abstract}
\section{Introduction}
The Askey-Wilson divided difference operator is defined by
\begin{align}\label{0.3}
\mathcal{D}_q\,f(x)=\frac{\breve{f}\big(q^{1/2} e^{i\theta}\big)
-\breve{f}\big(q^{-1/2} e^{i\theta}\big)}{\breve{e}\big(q^{1/2}e^{i\theta}\big)-\breve{e}\big(q^{-1/2} e^{i\theta}\big)},
\end{align}
where, for each polynomial $f$,  $\breve{f}(e^{i\theta})=f(\cos \theta)$ and $e(x)=x$ (see \cite[Section 12.1]{I05}). In \cite{CMP}, we give positive answer to the first part of the following Ismail's conjecture (see \cite[Conjecture 24.7.8]{I05}):
\begin{conj}
Let $(p_n)_{n\geq0}$ be a sequence of orthogonal polynomials and let $\pi$ be a polynomial which does not depend on $n$. If
\begin{align*}
\pi(x)\mathcal{D}_q\,p_n (x)=\sum_{k=-1}^1 c_{n, k} p_{n+k}(x),
\end{align*}
then $(p_n)_{n\geq0}$ are continuous $q$-Jacobi polynomials or Al-Salam-Chihara polynomials, or special or limiting cases of them. The same conclusion follows if 
\begin{align}\label{0.2Dq-general}
\pi(x)\mathcal{D}_q\,p_n (x)=\sum_{k=-r}^s c_{n, k} p_{n+k}(x),
\end{align}
for positive integers $r$ and $s$.
\end{conj}

The second part of this conjecture is certainly a much more complex problem than the first one.  However, after much manipulation of similar structural relations in a number of recent works, we found the second part of the conjecture less and less convincing. Now in the next section we propose a counterexample.
\section{Counterexample}

Throughout this section we assume that $0<q<1$. Set $x(s)=(q^{s}+q^{-s})/2$. Taking $e^{i\theta}=q^s$ in \eqref{0.3}, $\mathcal{D}_q$ reads
\begin{equation*}
\mathcal{D}_q f(x(s))= \frac{f\big(x(s+\frac{1}{2})\big)-f\big(x(s-\frac{1}{2})\big)}{x(s+\frac{1}{2})-x(s-\frac{1}{2})}.
\end{equation*}
Define
\begin{align*}
\mathcal{S}_q f(x(s))=\frac{f\big(x(s+\frac{1}{2})\big)+f\big(x(s-\frac{1}{2})\big)}{2},
\end{align*}
and
\begin{align*}
\alpha_n&= \frac{q^{n/2}+q^{-n/2}}{2},\quad  \gamma_n=\frac{q^{n/2}-q^{-n/2}}{q^{1/2}-q^{-1/2}} \quad (n=0,1,\dots),
\end{align*}
and $\alpha=\alpha_1$.
Recall that
\begin{align}
\mathcal{D}_q \big(fg\big)&= \big(\mathcal{D}_q f\big)\big(\mathcal{S}_q g\big)+\big(\mathcal{S}_q f\big)\big(\mathcal{D}_q g\big), \label{def-Dx-fg} \\[7pt]
\mathcal{S}_q \big( fg\big)&=\big(\mathcal{D}_q f\big) \big(\mathcal{D}_q g\big)\texttt{U}_2  +\big(\mathcal{S}_q f\big) \big(\mathcal{S}_q g\big), \label{def-Sx-fg}
\end{align}
where $\texttt{U}_2(x)=(\alpha^2 -1)(x^2-1)$. All these properties and definitions, even the notation, can be found, for instance, in  \cite{CMP1}. The Askey-Wilson polynomials are defined by
$$
p_n(x; a, b, c, d\,|\, q)=a^{-n}\,(ab, ac, ad; q)_n\, \pPq{4}{3}{q^{-n}, a b c b d q^{n-1}, a e^{i \theta},  a e^{-i \theta}}{ab, ac, ad}{q, q},
$$
where $x=\cos \theta$. If we take $a=q^{1/2\alpha+1/4}$, $b=q^{1/2\alpha+3/4}$, $c=-a$, and $d=-b$, we get the continuous $q$-Jacobi polynomials. If we take $d=0$, we get the continuous dual $q$-Hahn polynomials. If we take $c=d=0$, we get the Al-Salam-Chihara polynomials. The sequence of monic continuous dual $q$-Hahn polynomials, $(H_n(x;a,b|q))_{n\geq 0}$, satisfies
\begin{align}\label{H}
xH_n(x;a,b,c|q)=H_{n+1}(x;a,b,c|q)+a_nH_{n}(x;a,b,c|q)+b_nH_{n-1}(x;a,b,c|q)\;,
\end{align}
where $H_{-n-1}(\cdot;a,b,c|q)=0$ and
\begin{align*}
a_n&=\frac12(a+a^{-1}-a(1-q^n)(1-bcq^{n-1}) -a^{-1}(1-abq^n)(1-acq^n)),\\[7pt]
 b_n&=\frac14(1-abq^{n-1})(1-acq^{n-1})(1-bcq^{n-1})(1-q^{n}).
\end{align*}
Define $P_n=H_{n}(\cdot ;1,-1,q^{1/4}|q^{1/2})$. Clearly, $P_n$ is not a continuous $q$-Jacobi polynomial or Al-Salam-Chihara polynomial or, much less, special or limiting cases of them. These polynomials satisfy, among other relations, a relation of type \eqref{0.2Dq-general} with $r=2$ and $s=1$.

\begin{proposition}\label{P}
Let $P_n(x)=H_{n}(x;1,-1,q^{1/4}|q^{1/2})$. The sequence $(P_n)_{n\geq 0}$ satisfies the following relations: 
\begin{align}
\label{main-eq1}\mathcal{S}_qP_n(x)&=\alpha_nP_n(x)+c_nP_{n-1}(x),\\[7pt]
\label{main-eq2} \texttt{U}_2(x) \mathcal{D}_qP_n(x)&=(\alpha^2-1)\gamma_nP_{n+1}(x)+\big(c_{n+1}-\alpha c_n +(1-\alpha)\alpha_n B_n\big)P_n(x)\\[7pt]
&\quad +\big((B_n-\alpha B_{n-1})c_n +(1-\alpha^2)\gamma_n C_n \big)P_{n-1}(x)\nonumber\\[7pt]
&\quad +(c_{n-1}C_n-\alpha c_nC_{n-1})P_{n-2}(x),\nonumber
\end{align}
where
\begin{align*}
B_n&=\frac{1}{2}\Big((1+q^{-1/2})q^{n/2} +1-q^{-1/2}\Big)q^{(2n+1)/4},\\[7pt]
C_n&=\frac{1}{4}(1+q^{(n-1)/2})(1-q^{n/2})(1-q^{n-1/2}),\\[7pt]
c_n&=C_n\, q^{-(2n-1)/4}.
\end{align*}
\end{proposition}
\begin{proof}
\eqref{H} makes it obvious that
\begin{align}
x P_n(x)=P_{n+1}(x)+B_nP_n(x)+C_nP_{n-1}(x) \quad (n=0,1,\dots).\label{TTRR}
\end{align}
The proof is by complete mathematical induction on $n$. Note that $c_0=0$ and $\alpha_0=1$, and so
$$
c_1-\alpha c_0+(1-\alpha)\alpha_0 B_0=0.
$$
Hence, for $n=0$ we have
\begin{align*}
\mathcal{S}_qP_0&=1=\alpha_0P_0 +c_0P_{-1}=1,\\[7pt]
\texttt{U}_2(x)\mathcal{D}_qP_0(x)&=0=(\alpha^2-1)\gamma_0P_1 +0P_0=0.
\end{align*}
Assuming \eqref{main-eq1} and \eqref{main-eq2} hold for all $n\leq k$, we will prove it for $n=k + 1$. Set
\begin{align*}
c_{n, 1}&=(\alpha^2-1)\gamma_n,\\[7pt]
 c_{n, 2}&=c_{n+1}-\alpha c_n +(1-\alpha)\alpha_n B_n,\\[7pt]
c_{n, 3}&= (B_n-\alpha B_{n-1})c_n +(1-\alpha^2)\gamma_nC_n, \\[7pt]
c_{n, 4}&= c_{n-1}C_{n} -\alpha c_nC_{n-1}.
\end{align*}
Now \eqref{main-eq2} reads as 
\begin{align}\label{aux2}
\texttt{U}_2(x) \mathcal{D}_qP_n(x)=c_{n, 1}P_{n+1}(x)+c_{n, 2}P_n(x)+c_{n, 3}P_{n-1}(x)+c_{n, 4}P_{n-2}(x).
\end{align}
Applying $\mathcal{S}_q$ to \eqref{TTRR}, and using \eqref{def-Sx-fg}, we get
\begin{align}
\nonumber \mathcal{S}_qP_{n+1}(x)&=\texttt{U}_2(x)\mathcal{D}_q\,x \mathcal{D}_qP_n(x)+\mathcal{S}_q\,x \mathcal{S}_qP_{n}(x)-B_n\mathcal{S}_qP_{n}(x)+C_n \mathcal{S}_qP_{n-1}(x)\\[7pt]
\label{aux1}&=\texttt{U}_2(x)\mathcal{D}_qP_n(x)- B_n \mathcal{S}_qP_n(x)-C_n\mathcal{S}_qP_{n-1}(x) + \alpha x \mathcal{S}_qP_n(x). 
\end{align}
From \eqref{aux1}, and using \eqref{aux2} for $n=k$ and \eqref{main-eq1} for $n=k-1$ and $n=k$, we obtain
\begin{align*}
\mathcal{S}_qP_{k+1}(x)&= c_{k, 1}P_{k+1}(x)+c_{k, 2}P_k(x)+c_{k, 3}P_{k-1}(x)+c_{k, 4}P_{k-2}(x)\\[7pt]
&\quad -B_k(\alpha_k P_k(x)+c_k P_{k-1}(x))-C_k (\alpha_{k-1} P_{k-1}(x)+c_{k-1} P_{k-2}(x))\\[7pt]
&\quad +\alpha x (\alpha_k P_k(x)+c_k P_{k-1}(x)).
\end{align*}
From \eqref{TTRR} we have 
\begin{align*}
\alpha x (\alpha_k P_k(x)+c_k P_{k-1}(x))&=\alpha \alpha_k(P_{k+1}(x)+B_k P_k(x)+C_kP_{k-1}(x))\\[7pt]
&\quad +\alpha c_k (P_k(x)+B_{k-1} P_{k-1}(x)+C_{k-1} P_{k-2}(x)).
\end{align*}
We leave to the reader the verification that 
$$
c_{k,3} +(\alpha \alpha_k-\alpha_{k-1}) C_k+(\alpha B_{k-1}-B_k)c_k=c_{k,4} +\alpha c_k C_{k-1}-c_{k-1}C_k=0,
$$
and 
$$
\alpha_{k+1}=c_{k, 1}+\alpha\, \alpha_k, \quad c_{k+1}=c_{k, 2}+\alpha c_k+(\alpha-1)\alpha_k B_k.
$$
We thus get
\begin{align*}
\mathcal{S}_qP_{k+1}(x)&=(c_{k, 1}+\alpha\, \alpha_k)P_{k+1}(x)+(c_{k, 2}+\alpha c_k+(\alpha-1)\alpha_k B_k)P_k(x)\\[7pt]
&\quad + (c_{k,3} +(\alpha \alpha_k-\alpha_{k-1}) C_k+(\alpha B_{k-1}-B_k)c_k) P_{k-1}(x)\\[7pt]
&\quad +(c_{k,4} +\alpha c_k C_{k-1}-c_{k-1}C_k)P_{k-2}(x)\\[7pt]
&=\alpha_{k+1}P_{k+1}(x)+c_{k+1}P_k(x),
\end{align*}
and \eqref{main-eq1} holds for $n=k+1$. Applying now $\mathcal{D}_q$ to \eqref{TTRR}, and using \eqref{def-Dx-fg}, we get
\begin{align}\label{aux3}
\mathcal{D}_q P_{k+1}(x)=\mathcal{S}_qP_{k}(x)+(\alpha x-B_k)\mathcal{D}_qP_k(x)-C_k\mathcal{D}_q P_{k-1}(x).
\end{align}
From \eqref{TTRR} we have 
\begin{align*}
\texttt{U}_2(x) P_k(x)&=(\alpha^2-1) P_{k+2}(x)+(\alpha^2-1)(B_k+B_{k+1})P_{k+1}(x)\\[7pt]
&\quad +(\alpha^2-1) (B_k^2+C_{k+1}+C_k-1)P_k(x)\\[7pt]
& \quad +(\alpha^2-1) C_k (B_k+B_{k-1})P_{k-1}(x)+(\alpha^2-1) C_kC_{k-1}P_{k-2}(x).
\end{align*}
Hence, multiplying \eqref{aux3} by $\texttt{U}_2$  and using \eqref{main-eq1} for $n=k$ and \eqref{aux2} for $n=k-1$ and $n=k$, we get
\begin{align*}
\texttt{U}_2(x)\mathcal{D}_q P_{k+1}(x)&= d_{k, 1}P_{k+2}(x)+d_{k, 2}P_{k+1}(x)+d_{k, 3}P_{k}(x)+d_{k, 4}P_{k-1}(x)\\[7pt]
&\quad +d_{k, 5}P_{k-2}(x)+d_{k, 6}P_{k-3}(x),
\end{align*}
where
\begin{align*}
d_{k, 1}&= (\alpha^2-1)\alpha_k +\alpha c_{k, 1},\\[7pt]
 d_{k, 2}&=(\alpha^2-1)(c_k+\alpha_k (B_k+B_{k+1}))+\alpha c_{k, 2}-(B_k-\alpha B_{k+1})c_{k, 1},\\[7pt]
 d_{k, 3}&=(\alpha^2-1)((B_k+B_{k-1})c_k +\alpha_k(B_{k} ^2 +C_{k}+C_{k+1}-1)) +\alpha c_{k, 1} C_{k+1}\\[7pt]
&\quad -c_{k-1,1}C_k +(\alpha-1)c_{k, 2}B_n +\alpha c_{k, 3},\\[7pt]
d_{k, 4}&=(\alpha^2-1)((B_k+B_{k-1})\alpha_k C_k+(C_k+B_{k-1} ^2 +C_{k-1}-1)c_k)\\[7pt]
&\quad -(c_{k-1,2}-\alpha c_{k, 2})C_k-(B_k-\alpha B_{k-1})c_{k,3}+\alpha c_{k, 4},\\[7pt]
d_{k, 5}&=(\alpha^2-1)C_{k-1}(\alpha_k C_k +c_k(B_{k-1}+B_{k-2}))+\alpha c_{k, 3}C_{k-1}-c_{k-1,3}C_k\\[7pt]
&\quad -(B_k-\alpha B_{k-2})c_{k, 4},\\[7pt]
d_{k, 6}&= (\alpha^2-1)c_k C_{k-1}C_{k-2}+\alpha c_{k, 4}C_{k-2}-c_{k-1,4}C_k.
\end{align*}
Finally, the reader should satisfy himself that $d_{k, 1}=c_{k+1,1}$, $d_{k, 2}=c_{k+1,2}$, $d_{k, 3}=c_{k+1,3}$, $d_{k, 4}=c_{k+1,4}$, $d_{k, 5}=0$, and  $d_{k, 6}=0$, and \eqref{main-eq2} holds for $n=k+1$, but this is easily verified using any mathematical software. That completes the inductive step, and hence the proof.
\end{proof}

The previous proof by induction was elaborated after having found the sequence $(P_n)_{n\geq 0}$ by a method similar to the one used in \cite{CMP}, which involves  a high degree of technical complexity. However, as Galileo reputedly said: ``All truths are easy to understand once they are discovered. The point is to discover them". In this sense, and view of Proposition \ref{P}, an interesting open problem is to characterize the sequence of orthogonal polynomials $(p_n)_{n\geq 0}$ such that $\mathcal{S}_q p_n$ can be written as a linear combination of $p_n$ and $p_{n-1}$. 

\section*{Acknowledgements}
This work is supported by the Centre for Mathematics of the University of Coimbra, funded by the Portuguese Government through FCT/ MCTES. DM is partially supported by ERDF and Consejer\'ia de Econom\'ia, Conocimiento, Empresas y Universidad de la Junta de Andaluc\'ia (grant UAL18-FQM-B025-A) and by the Research Group FQM-0229 (belonging to Campus of International Excellence CEIMAR).

 \end{document}